%% file: Genus_5.tex
\documentclass[12pt]{amsart}
\usepackage[utf8]{inputenc} 
\usepackage{url}
\usepackage{hyperref}
\usepackage{amsmath}
\usepackage{mathtools}
\usepackage{faktor} 
\usepackage{amssymb}
\usepackage{mathrsfs}
\usepackage{amsthm}
\usepackage{graphicx}
\usepackage{color}
\usepackage[top=2cm,bottom=2cm,left=2cm,right=2cm]{geometry}
\usepackage[all]{xy}
\usepackage[titletoc]{appendix}
\usepackage{ulem}
\usepackage{csquotes}

\newcommand{\id}{\operatorname{id}} 

\newtheorem{thm}[subsection]{Theorem}
\newtheorem{lem}[subsection]{Lemma}

\newtheorem{pro}[subsection]{Proposition}

\newtheorem{defn}[subsection]{Definition}
\newtheorem{rem}[subsection]{Remark}

\begin{document}

\title{Maximal curves of genus 5 over finite fields } 

\author[]{Leolin~Nkuete}
\address{Department of Mathematics, University of Luxembourg, Maison du nombre, 6~avenue de la Fonte, L-4364 Esch-sur-Alzette,
Luxembourg} \email{leolin.nkuete@uni.lu}

\author[]{Antigona~Pajaziti}
\address{Department of Mathematics, University of Luxembourg, Maison du nombre, 6~avenue de la Fonte, L-4364 Esch-sur-Alzette,
Luxembourg}
\address{Mathematical Institute, Leiden University, Niels Bohrweg 1, 2333 CA Leiden, The Netherlands}
\email{antigona.pajaziti@uni.lu\\
a.pajaziti@math.leidenuniv.nl}

\author[]{Hamide~Suluyer}
\address{Department of Mathematics, Atılım University, 06830 Gölbaşı, Ankara, Turkey}
\email{hamide.kuru@atilim.edu.tr}

\author[]{Rabia~G\"ul\c{s}ah~Uysal \\  with an Appendix by René Schoof}
\address{Department of Mathematics, Middle East Technical University, Odtü Matematik Bölümü Z-35 Çankaya, Ankara, Turkiye }
\email{guysal@metu.edu.tr}

\date{}

\begin{abstract}
A maximal curve over a finite field $\mathbb F_q$  is a curve whose number of points reaches the upper Hasse--Weil--Serre bound. We define the discriminant of $\mathbb F_q$ as $d(\mathbb F_q):= \lfloor2\sqrt{q}\rfloor^2-4q$, which arises as the discriminant of the characteristic polynomial of the Frobenius for a maximal elliptic curve defined over $\mathbb F_q$. In this article, we investigate the existence of a maximal curve of genus $5$ defined over a finite field $\mathbb F_q$ of discriminant $-19$. Using the knowledge of the automorphism group of such a curve, we prove that such  a curve does not exist when $q\equiv 2,3,4 \mod 5$. In the case $q\equiv 1\mod 5$ we give models of the potential maximal curve. Finally, for the case $q\equiv 0\bmod 5$, we prove that such a curve might exist only for $q=5^7$. 
\end{abstract}

\maketitle
\section{Introduction}
Throughout this paper, by a curve we mean a smooth, projective, absolutely irreducible algebraic variety of dimension 1. Let $\mathbb F_q$ be a finite field, where $q=p^r$ for some integer $r\geq1$ and some prime number $p$. 
Given a curve $C$ defined over $\mathbb F_q$, the following Hasse--Weil--Serre bound in \cite{Perret1991} gives an upper bound on the number of the $\mathbb{F}_q$-rational points on $C$ $$\# C(\mathbb{F}_q)\leq g\lfloor 2\sqrt{q}\rfloor+q+1,$$ where $g$ is the genus of the curve $C$. When the number of $\mathbb F_q$-rational points reaches the upper Hasse-Weil-Serre bound $q+1+ g\lfloor 2\sqrt{q}\rfloor$, we say that $C$ is a maximal curve. Maximal  curves play an important role in coding theory. It is well known that curves with small genus $g$ and a large number of rational points  generate ``good''  algebraic geometry codes, as presented by Goppa in \cite{Goppa1981b}. One might ask, given fixed $q$ and $g$: Can one find a maximal  curve of genus $g$ defined over $\mathbb F_q$?  This kind of question has been studied in the literature and 
in 2010, van der Geer, Howe, Lauter and Ritzenthaler created the website \href{https://manypoints.org/}{manypoints.org} \cite{manypoint}, where there are listed the best known upper and lower bounds for the maximal number of rational points for curves of genus $g$ defined over a finite field $\mathbb F_q$. However, knowing a bound on the maximal number of rational points does not guarantee the existence of a curve  whose number of rational points reaches the bound. 

We let $d(\mathbb F_q):= \lfloor2\sqrt{q}\rfloor^2-4q$ be the discriminant of $\mathbb F_q$. This quantity happens to be the discriminant of the characteristic polynomial of the Frobenius for a maximal elliptic curve defined over $\mathbb F_q$, as will be seen in Section \ref{section 2}. Previous works on maximal curves of genus $1,2,3$ and $4$ over a finite field of discriminant $-19$ can be found in \cite{AlekseenkoOptimal} and \cite{zaytsevoptimal}.
In this paper we will investigate the existence of a maximal curve of genus $5$ defined over $\mathbb F_q$ with $ d(\mathbb F_q)=-19$.  In particular, the latter condition implies $q\equiv -x^2+1\bmod 5$, where $x=\lfloor2\sqrt{q}\rfloor$. It follows that the cases $q\equiv 3,4\bmod 5$ do not occur, so we are restricted to the cases $q\equiv 0,1, 2\bmod 5$.  

By the theorem of Enrique - Babbage, see \cite[Proposition 2.8, Chapter VII]{Miranda}, a curve of genus $5$ can be either hyperelliptic, trigonal or a complete intersection of three quadrics in $\mathbb{P}^4$. A curve $C$ is called trigonal if there exists a morphism from $C$ to $\mathbb{P}^1$ of degree $3$. A trigonal curve of genus $5$ can be represented as a plane quintic with a node or cusp,  as explained in Section \ref{sec: q=2[5]}. In \cite[Lemma 4.11]{zaytsevoptimal},
Zaytsev showed that if the discriminant $d(\mathbb F_q)=-19$, a maximal curve  of genus $5$ over $\mathbb F_q$ cannot be hyperelliptic. Therefore, under the same condition,  we are left to investigate the cases when the curve is trigonal or a complete intersection of three quadrics in $\mathbb P^4$. In section \ref{section 2}, using Serre's equivalence of categories between the category of ordinary principally polarized abelian varieties over $\mathbb F_q$ isogenous to a power of an ordinary elliptic curve, and the category of unimodular irreducible hermitian torsion free modules and following Zaytsev \cite{zaytsevoptimal}, we describe the automorphism group of a maximal curve of genus $5$. In Section \ref{sec: q=2[5]}, we investigate the existence of a maximal curve of genus $5$ over $\mathbb F_q$ with  $d(\mathbb F_q)=-19$  and $q  \equiv 2 \bmod 5$. We recall in Section \ref{sec: proj} the theory of projective representations of a finite group $G$ defined over a field $K$, which we use later in Section \ref{sec: q=1[5]}. In Section \ref{sec: q=1[5]}, we first describe a potential equation of a maximal trigonal curve of genus $5$ defined over $\mathbb F_q$ with $q\equiv 1 \bmod{5 }$. Then, in the case where the curve arises as a complete intersection of three quadrics in $\mathbb P^4$, we apply the framework introduced in Section \ref{sec: proj} to the automorphism group of the curve to describe potential equations of a maximal curve of genus $5$ defined over $\mathbb F_q$ with $q\equiv 1 \bmod{5}$.

Finally in Section \ref{sec: q=0[5]}, for $q\equiv 0\mod 5$ we prove that, except for  $q=5^7$,   $q \equiv  0\bmod 5$ is not compatible with the condition $d(\mathbb F_q)=-19$. 

The main results of this paper are the following:

\begin{thm}
  There is no maximal curve of genus $5$ over $\mathbb F_q$ when $q \equiv 2,3,4 \bmod 5$ and $d(\mathbb F_q)=-19$.
\end{thm}

\begin{thm}
    The only value of $q$ satisfying $q\equiv 0 \bmod 5$  and $d(\mathbb{F}_q)=-19$ is $q=5^7$. 
\end{thm}

\section{Automorphism group of maximal curves of genus $5$ over finite fields with $d(\mathbb F_q)=-19$}\label{section 2}

 Let $C$ be a maximal curve of genus $g$ defined over $\mathbb F_q$. The characteristic polynomial of the Frobenius endomorphism $\pi$ of the Jacobian $\operatorname{Jac(C)}$ can be written  as $$
f(x)=\prod_{i=1}^{2g}(x-\alpha_i)
\in \Bar{\mathbb{Q}}[x].$$ The trace of $\pi$ is given by  $\operatorname{Tr}(\pi)=\sum_{i=1}^g (\alpha_i+\overline{\alpha_i})$, where $\overline{\alpha_i}=\alpha_{g+i}$, see \cite[Proposition 2.1.2 ]{serre2020rational}.
 As $C$ is maximal, we have  $$\#C(\mathbb F_q)=1+q + g\lfloor 2 \sqrt{q}\rfloor=1+q - \sum_{i=1}^g (\alpha_i+\overline{\alpha_i}).$$
Hence, $\operatorname{Tr}(\pi)=- g\lfloor 2 \sqrt{q}\rfloor$. Then by \cite[Theorem 2.1.4]{serre2020rational}, we have that $\alpha_i+ \overline{\alpha_i}=-\lfloor 2 \sqrt{q}\rfloor$.
Moreover, since  $\alpha_i \cdot \overline{\alpha_i}=q$, it follows that for each $i$, the zeroes $\alpha_i$ and $\overline{\alpha_i}$ are the roots of  the polynomial $$x^2 +\lfloor 2 \sqrt{q}\rfloor x+q.$$
Therefore,  the characteristic polynomial of $\operatorname{Jac(C)}$  splits into $g$ copies of $x^2 +\lfloor 2 \sqrt{q}\rfloor x+q$, which is the characteristic polynomial of an elliptic curve $E$ defined over $\mathbb F_q$. Furthermore, if $E$ is an ordinary elliptic curve (which is the case when $d(\mathbb F_q)=-19$, see \cite[Proposition 3.1]{zaytsevoptimal}), then E is maximal, see \cite[Theorem 2.6.3]{serre2020rational}.
  
By the Honda-Tate theorem about isogeny classes of simple abelian varieties, see \cite[Theorem 16.1]{Milne}, it follows that $\operatorname{Jac(C)}$ is isogenous to $E^g$. Also since $E$ is an ordinary elliptic curve, the Jacobian $\operatorname{Jac(C)}$ is an ordinary principally polarized abelian variety.

\begin{rem}
        The discriminant of the polynomial $x^2 +\lfloor 2 \sqrt{q}\rfloor x+q$ coincides with the quantity $d(\mathbb F_q)$ introduced above. It follows that  the Frobenius of $E$ generates the imaginary quadratic  number field $\mathbb{Q}(\sqrt{d(\mathbb F_q)})$. 
\end{rem}

A key tool to study curves over finite fields is Serre's equivalence of categories between the category of ordinary principally polarized abelian varieties over $\mathbb F_q$ isogenous to $E^g$ for some ordinary elliptic curve $E$, and the category of unimodular irreducible  hermitian $ R$-modules without torsion of rank $g$, described in the Appendix by J.-P. Serre in \cite{Lauter} and in \cite{serre2020rational}. To be more precise, the isomorphism classes of the abelian varieties defined over $\mathbb{F}_q$  in the isogeny class of $E^g$  correspond to  the isomorphism classes of $R$-modules that can be embedded as a lattice in the $K$-vector space $K^g$, where $R=\operatorname{End}_{\mathbb{F}_q}(E)$ and $K=\operatorname{Quot}(R)$. If the quotient $\mathbb{Z}[x]/(x^2 +\lfloor 2 \sqrt{q}\rfloor x+q)$ is equal to $\mathcal{O}_K$,  then  $R=\mathcal{O}_K$, where $\mathcal{O}_K$ is the ring of integers of the imaginary  quadratic number field $K $ of discriminant $d(\mathbb F_q)$. Furthermore, when the class number  $\# \operatorname{Cl}(\mathcal{O}_K)$ is equal to one, there is only one isomorphism class of such $\mathcal{O}_K$-modules  and the equivalence of categories  tells us that  the unique isomorphism  class in the isogeny class of $E^g$ corresponds to the unique free $\mathcal{O}_K$-module $\mathcal{O}_K^g$ and the canonical polarization corresponds to a unimodular irreducible  hermitian form $h: \mathcal{O}_K^g \times \mathcal{O}_K^g \rightarrow \mathcal{O}_K$.  

 Hence, applying the equivalence of categories to    $\operatorname{Jac}(C)$ with  $d(\mathbb F_q)$ such that $\mathcal{O}_K$ has class number one, we have that $\operatorname{Jac}(C)$ is isomorphic to $E^g$ over $\mathbb F_q$ and $$\operatorname{Aut}_{\mathbb F_q}(\operatorname{Jac}(C))\cong \operatorname{Aut(\mathcal{O}_K^g},h)$$ where, $\operatorname{Aut}(\mathcal{O}_K^g,h)$ is the automorphism group of the hermitian module $\mathcal{O}_K^g$ in $(K^g,h)$. Then computing  $\operatorname{Aut}_{\mathbb F_q}(\operatorname{Jac}(C))$ reduces to computing the automorphism group of the hermitian module $(\mathcal{O}_K^g,h)$. 
There exists a classification given by A. Schiemann in \cite{Schiemann} of unimodular irreducible hermitian modules together with the order of their automorphisms group for certain values of $d(\mathbb F_q)$ and $g$. The computations of Schiemann's results were made available online by R. Schulze-Pillot in \cite{hermitian}.  Using Torelli's theorem in \cite[Theorem 2.5.9]{serre2020rational}, we can deduce that the order of the automorphism group of the curve as it  follows:  
\begin{align*}
\#\operatorname{Aut}_{\mathbb{F}_q}(\operatorname{Jac}(C))
&= \#\operatorname{Aut}(\mathcal{O}_K^g, h)=
\begin{cases}
\#\operatorname{Aut}(C) & \text{if $C$ is hyperelliptic}, \\
2 \times \#\operatorname{Aut}(C) & \text{otherwise}.
\end{cases}
\end{align*}
\begin{rem}\label{rem aut}
All automorphisms of the curve $C$ are defined over $\mathbb F_q$,  for $d(\mathbb F_q)=-19$ . Indeed, $$\operatorname{Aut}_{\overline{\mathbb F}_q}(C)\subseteq\operatorname{Aut}_{\overline{\mathbb F}_q}(\operatorname{Jac}(C))=\operatorname{Aut}_{\overline{\mathbb F}_q}(E^g)\subseteq \operatorname{End}_{\overline{\mathbb F}_q}(E^g)=\operatorname{Mat}_g(\operatorname{End}_{\overline{\mathbb F}_q}(E)).
$$
    We have that $\mathbb{Z}[x]/(x^2 +\lfloor 2 \sqrt{q}\rfloor x+q) \subset \operatorname{End}_{\mathbb{\overline{F
     }}_q}(E) \subset \mathcal{O}_K$, and since $d(\mathbb F_q)$ is square free and  $d(\mathbb F_q)\equiv 1\bmod 4$, we have   $\mathbb{Z}[x]/(x^2 +\lfloor 2 \sqrt{q}\rfloor x+q)=\mathcal{O}_K$. It follows that    $\operatorname{Mat}_g(\operatorname{End}_{\overline{\mathbb F}_q}(E))=\operatorname{Mat}_g(\mathcal{O}_K)=\operatorname{Mat}_g(\operatorname{End}_{\mathbb F_q}(E))$, where $\operatorname{Mat}_g$ is the algebra of $g\times g$ matrices.
  
\end{rem}

Following the discussion above and looking at the generators of automorphism groups of unimodular irreducible hermitian modules for $g=5$, Zaytsev showed that:  

\begin{thm}[Theorem 4.13 \cite{zaytsevoptimal}]\label{zaytsev}
If there exists a maximal curve of genus $5$ over a finite field $\mathbb{F}_{q}$ with $d(\mathbb F_q)=-19$, then the automorphism group of the curve is isomorphic to the Dihedral group $D_5$.
\end{thm}

From now on, we will investigate the existence of a maximal curve $C$ of genus $5$ over $\mathbb F_q$ with $d(\mathbb F_q)=-19$ and $\operatorname{Aut}(C)\cong D_5$. We consider the cases where $C$ is a trigonal curve or a complete intersection of three quadrics in $\mathbb P^4$.

\section{Maximal curves of genus $5$ over $\mathbb F_q$ with $d(\mathbb F_q)=-19$ and $q  \equiv 2 \bmod 5$ }\label{sec: q=2[5]}

 Let $C$ be a curve of genus $5$. If $C$ is not hyperelliptic, then it is canonically described as an intersection of three quadrics in $\mathbb{P}^4$. This intersection can be complete or not. By Petri’s theorem \cite[p.131] {Arbarello}, in the complete intersection case, the homogeneous ideal of the curve is generated by three quadrics.
 
If the intersection is not complete, then the curve is trigonal, that is,  there exists a degree $3$ morphism $\phi: C \to \mathbb{P}^{1}$, see \cite[p. 118]{Arbarello}.
For any unramified $P \in \mathbb{P}^1$, we have the degree $3$ divisor $D=\phi^{-1}(P)$ on $C$, which we can write effectively as $D=P_{1}+P_{2}+P_{3}$. Since $D$ has a positive degree, the dimension $\ell(D)$ of the Riemann-Roch space $\mathscr{L}(D)$ associated to the divisor $D$ is greater than or equal to $2$. Then by the Riemann-Roch theorem,
$$
\ell(\kappa-D)=\ell(D)-3+5-1 \geqslant 3
$$ where $\kappa$ is the canonical divisor. This tells us that we must have at least a three dimensional family of hyperplanes passing through $P_1, P_2$ and $P_3$. Indeed, the elements of $\mathscr{L}(\kappa-D)$ are sections in $\mathscr{L}(\kappa)$ with zeros on $D$. This imposes that the points must be collinear in $\mathbb P^4$. It follows that the dimension of the space $\operatorname{span}(D)$ generated by the intersection of all hyperplanes containing $D$  has dimension $1$. By the geometric Riemann Roch formula (see \cite[Theorem 2.7, p.208]{Miranda}), we have $\ell(D)=\deg D-\dim\operatorname{span}(D)=2$.  Therefore $\ell(\kappa-D)=3$.
Then the complete linear system $|\kappa-D|$ gives a morphism $\varphi_{\kappa-D}: C\rightarrow \mathbb{P}^{2}$  of degree $5$. Let $C'$ be the image of $C$ in $\mathbb{P}^2$. Then $C'$ is a plane curve of degree $5$ and by the genus formula for plane curves, $C'$ has one singular point which is a node.

Remark that the classification by Enrique-Babbage of a curve of genus $5$ also holds over $\mathbb F_q$. The complete intersection case holds over $\mathbb F_q$ (not only over $\overline{\mathbb F}_q$) because the canonical divisor $\kappa$, which gives an embedding of $C$ in $\mathbb P^4$ is defined over $\mathbb F_q$. In the trigonal case as well, the morphism $ C\rightarrow \mathbb{P}^{1}$ is defined over $\mathbb F_q$, see \cite[Section 2.1]{Kudo}.


The following lemma shows that the automorphism group of the trigonal curve $C$ is exactly the linear automorphism group of the singular plane model curve $C'$.

\begin{lem}[ Lemma 2.1.2, \cite{Kudo}] \label{automorphismtrigonal}
Let $C$ be a trigonal curve of genus $5$. There is a canonical isomorphism from $\operatorname{Aut}(C)$ to $\{\varphi \in \operatorname{Aut}(\mathbb P^2) \mid \varphi(C')=C'\}$.
\end{lem}

\begin{thm}\label{Theorem 3.2}
There is no maximal curve of genus $5$ over $\mathbb F_q$, where $q  \equiv 2 \bmod 5$ and $d(\mathbb F_q)=-19$. 
\end{thm}

\begin{proof}
Let $C$ be a maximal curve of genus $5$ defined over $\mathbb F_q$. If $C$ is trigonal, Lemma \ref{automorphismtrigonal} and Theorem \ref{zaytsev} imply that $D_5=\operatorname{Aut}(C)$ is a subgroup of $\operatorname{PGL}_3(\mathbb F_q)$. However, $q \not\equiv 0,\pm 1 \bmod 5$ implies that $5$ does not divide $\# \operatorname{PGL}_3(\mathbb F_q)$. We conclude that the curve $C$ cannot be trigonal.

Now, assume $C$ is the complete intersection of three linearly independent quadrics $Q_1,Q_2$ and $Q_3$ in $\mathbb P^4$.  As $C$ is non-singular and non-hyperelliptic, $\operatorname{Aut}(C) \subset \operatorname{PGL}_5(\mathbb{F}_q)$ (see \cite[Theorem 1.3]{Nart}). Let $\varphi\in\operatorname{Aut}(C)$ be an element of order $5$. Let $V:= \langle Q_1, Q_2, Q_3 \rangle $ be the $\mathbb{F}_q$-vector space generated by the three quadrics, viewed as a subspace of the vector space generated by all quadrics monomials. Then $\varphi$ induces an automorphism $\Tilde{\varphi}\in \operatorname{Aut}(V)\subset \mathbb{P}(V)=\operatorname{PGL}_3(\mathbb{F}_q)$. The order of $\tilde{\varphi}$ is $1$ or $5$ but since $q \not\equiv 0,\pm 1 \bmod 5$ it follows that the order of $\tilde{\varphi}$  is $1$. Hence $\tilde{\varphi}$ acts trivially on the vector space $V$.
Let $M$ be the matrix associated to $\varphi$ in $\operatorname{PGL}_5(\mathbb{F}_q)$. Since  $5\nmid q-1$, there exists an equivalent matrix $M'$ such that $M'^5=\operatorname{Id}$.  Thus, $M'$ is equivalent to the companion matrix of the polynomial $X^5-1$.
Therefore, after a change of basis the action of $\varphi$ on the coordinate points of $C$ is given as follows $\varphi: [x:y:z:u:v]\rightarrow [y:z:u:v:x]$. Applying the action of $\varphi$ together with the fact that $\tilde{\varphi}$ acts trivially on $V$, we can write the quadrics in the following form:
$$Q_i=\lambda_i(x^2+y^2+z^2+u^2+v^2)+\mu_i(xy+yz+zu+uv+xv)+\vartheta_i(xz+yu+zv+xu+yv),$$
where $\lambda_i,\mu_i, \vartheta_i \in \mathbb F_q$. Let $P_1= x^2+y^2+z^2+u^2+v^2$, $P_2= xy+yz+zu+uv+xv$ and $P_3=xz+yu+zv+xu+yv$. These quadrics are linearly independent and the quadrics  $Q_{1}, Q_{2}, Q_{3}$ belong to the $\mathbb{F}_q$-vector space generated by $P_1,P_2$ and $P_3$.  The quadrics $P_1,P_2, P_3$ define over $\mathbb{F}_q$ a singular variety of dimension $1$ with a singular point $[1:\zeta: \zeta^2:\zeta^3:\zeta^4]\in \mathbb P(\mathbb{F}_{q^4}^5)$, where $\zeta$ is a primitive $5$-th root of unity.  Moreover, this variety is a closed subset of $\mathbb{P}^4$, contained  in the curve defined by the complete intersection of $Q_{1}, Q_{2}, Q_{3}$. It follows that the variety defined by  $Q_1,Q_2, Q_3$ is singular.  We conclude that a maximal curve of genus $5$ cannot be a complete intersection of three quadrics in $\mathbb P^4$.
\end{proof}



\section{Some projective representations}\label{sec: proj}
In this Section, we are interested in studying projective representations in terms of linear representations of central group extensions. We will use the theory of projective representations later in Section \ref{sec: q=1[5]} to determine a model of the genus $5$ curve in the complete intersection case.

Let $G$ be a finite group and $K$ be a field. To any projective representation 
$$\rho: G\longrightarrow \operatorname{PGL}(W),$$ where $W$ is a finite dimensional $K$-vector space, we can associate a map of sets $\rho': G\longrightarrow\operatorname{GL}(W)$ called a lift and a cohomology class $[\alpha]\in \operatorname{H}^2(G,K^\times)$ called the multiplier such that 
$$\rho'(g)\rho'(h)= \alpha(g,h)\rho'(gh)$$
for all $g,h\in G$, $\rho=\Pi\circ \rho'$, where $\Pi:  \operatorname{GL}(W)\longrightarrow \operatorname{PGL}(W)$. 

More precisely, for each $g\in G$, we define  $\rho'(g)$  as the unique element in a fixed set of coset representatives of $\operatorname{GL}(W)$ in $\operatorname{PGL}(W)$ such that $\Pi(\rho'(g))=\rho(g)$. Since for  $g,h\in G$, we have $$\rho'(gh)K^\times=\rho'(g)\rho'(h)K^\times,$$
there is a unique $\alpha(g,h)\in K^\times$ such that  $$\rho'(g)\rho'(h)=\alpha(g,h)\rho'(gh).$$
It turns out that $\alpha: G\times G\longrightarrow K^\times$ is a $2$-cocycle. 

In particular, the class of $\alpha$ is trivial if and only if  $\rho$ is projectively equivalent to a representation $\Pi\circ T$, where  $T:G\longrightarrow \operatorname{GL}(W)$ is a linear representation  (see \cite[Lemma 2.3.1]{zKarpilovsky}).  In this situation, we say 
that the projective representation $\rho$ of $G$ is liftable to a linear representation $T$ of $G$. 

We remark that for a given lift  $\rho'$ of $\rho$, there is a unique choice of $\alpha$ and then a unique class $[\alpha]\in \operatorname{H}^2(G,K^\times)$. There may be more than one lift associated to a projective representation. However, it is not difficult to show that different choices of lifts  will produce the same cohomology class in $\operatorname{H}^2(G,K^\times) $. Therefore, the cohomology class of $\rho$ is independent of the choice of a lift. 

Next we discuss the connection between projective representations and central extensions. 
\begin{defn}
 A central extension of $G$ is an exact sequence of groups $$\xymatrix{
          1\ar[r] &  Z \ar[r]^{i}& E \ar[r]^{p} & G\ar[r] &  1 
                    },$$
where  $Z$ is a subgroup of the center of $E$. We denote it $(E,p)$. When the group $Z$ is fixed, we say that  $E$ is a central extension of $G$ by $Z$.

We say that two extension $(E_1,p_1)$ and $(E_2,p_2)$ are equivalent if there is a group homomorphism $\varphi: E_1\longrightarrow E_2$ such that the following diagram commutes.
$$\xymatrixcolsep{3.5pc}\xymatrix{
  1\ar[r]& Z \ar[r] \ar[d]_{\id} & E_1 \ar[d]^{\varphi}  \ar[r]^{ p_1} &  G \ar[d]^{\id} \ar[r]& 1\\   
1\ar[r] & Z\ar[r] & E_2  \ar[r]^{p_2}   & G \ar[r]& 1 }$$
\end{defn}

\begin{thm}[Theorem 6.6.3 \cite{weibel1994introduction}]
The equivalence classes of extensions of $G$ by $Z$  are in one-to-one correspondence with the cohomology group $H^2(G, Z)$. In particular, the trivial class corresponds to the split extension.
\end{thm}

Let $Z$ be any subgroup of $K^\times$ seen as a trivial $G$-module and let $[\beta] \in H^2(G, Z)$.  We define a set
$$E_\beta:=G\times Z$$ and a multiplication $$(g,z)(h,t):=(gh, \beta(g,h)zt)$$ for all $(g,z), (h,t)\in E_\beta$.
Then $E_\beta$ with this multiplication is a group with the identity element $(1,1)$ and the inverse given by $(g,z)^{-1}=(g^{-1}, \beta(g,g^{-1})^{-1}z^{-1}).$ 
We define the maps $$\begin{array}{l}
     i: Z\longrightarrow E_\beta, \, z\mapsto (1,z ), \quad \text{ and}\quad
     
    p: E_\beta \longrightarrow G, \, (g,z)\mapsto g.
\end{array}$$
It is not difficult to see that they are group homomorphisms and 
$$ \xymatrix{
          1\ar[r] &  Z \ar[r]^{i}& E_\beta \ar[r]^{p} & G\ar[r] &  1 
                    }$$  is a central extension. 
The next proposition tells us when we can associate to a projective representation of $G$ a linear representation of a central extension of $G$. 
\begin{pro}[ Proposition 1.8, \cite{GE}]\label{prop:lift}
Let $\rho:G\longrightarrow \operatorname{PGL}(W) $ be a projective representation. Let  $\rho': G\longrightarrow \operatorname{GL}(W)$  and $\alpha\in Z^2(G, K^\times)$ be respectively a lift and the multiplier of $\rho$.  If there is a homomorphism $\varphi: Z\longrightarrow K^\times$, where $Z$ is a non trivial abelian group and $\beta\in Z^2(G, Z)$ such that $\varphi\circ \beta=\alpha$, then there is a  group homomorphism $T: E_\beta\longrightarrow \operatorname{GL}(W)$ such that the following diagram commutes: $$\xymatrixcolsep{3.5pc}\xymatrix{
  1\ar[r]& Z \ar[r] \ar[d]_{\varphi} & E_\beta \ar[d]^{T}  \ar[r]^{ p} &  G \ar[d]^{\rho} \ar[r]& 1\\   
1\ar[r] & K^{\times}\ar[r] & \operatorname{GL}(W)  \ar[r]^{\Pi}   & \operatorname{PGL}(W) \ar[r]& 1 }$$
\end{pro}

The Proposition \ref{prop:lift} also tells us that, for any projective representation of $G$ with multiplier $\alpha$, there is an associated linear representation of $E_\beta$. In particular, if the class of $\alpha$ is non-trivial, then the class of $\beta$ must be non-trivial since $\varphi$ is a group homomorphism. In this case, $E_\beta$ is a non-split extension.

\section{Maximal curves of genus $5$ over $\mathbb F_q$ with $d(\mathbb F_q)=-19$ and $q \equiv  1 \bmod 5$}\label{sec: q=1[5]}

In this section, we describe the potential equations of a maximal curve $C$ over $\mathbb F_q$ when it is trigonal or an intersection of three quadrics in $\mathbb{P}^4$ under the conditions $q \equiv 1 \bmod 5$ and $d(\mathbb F_q)=-19$.



Let $C$ be a maximal trigonal curve of genus $5$ defined over $\mathbb F_q$. Let $C'$ denote the singular plane model of $C$. Then $C'$ can be defined  by a homogeneous polynomial of degree $5$ in $\mathbb F_q[X,Y,Z]$ of the form
$$F=\sum_{\substack{0<i,j,k \leq 5 \\ i+j+k=5}} a_{ijk}X^i Y^j Z^k,$$
where $a_{ijk}\in \mathbb F_q$. Since $q\equiv 1 \bmod 5$, $\mathbb F_q$ contains a primitive $5$-th root of unity which we denote $\zeta$.

Recall that by Theorem \ref{zaytsev} the automorphism group of $C$ is $D_5$.
We aim to reduce the number of parameters of $F$ by applying the action of the automorphism group $\operatorname{Aut}(C)$ on the homogeneous polynomial $F$. Let $ \varphi$ and  $\sigma$ be the generators of $D_5$,  where $\varphi$ is of order $5$ and $\sigma$ is of order $2$. By Lemma \ref{automorphismtrigonal}, the automorphisms $\varphi$ and $\sigma$ induce  linear automorphisms $\varphi'$ and $\sigma'$, respectively, on $C'$, each of the same order as the original automorphisms. Since $q\equiv 1 \bmod 5$, after a suitable linear change of coordinates, $\varphi'$ can be given as 
$$\varphi': [X:Y:Z] \to [X:\zeta^nY:\zeta^mZ]$$ where  $1\leq n,m \leq 4.$ The condition that  $\varphi'$ is a generator of order $5$, together with the relation $\varphi'\sigma'\varphi'=\sigma'$, implies by a direct matrix computation that $m\neq n$ and $m+n\equiv 0\bmod 5$. Moreover, the condition $\sigma'^2=id$ implies that $\sigma'$ has the form
$$\sigma'=\begin{bmatrix}
    \pm1&0&0\\
    0&0&\zeta^m\\
    0&\zeta^n &0
\end{bmatrix},$$
where $\zeta^m=\zeta^{-n}$. Hence, the cases to consider are $(m,n)=(1,4)$ and $(m,n)=(2,3)$.    Applying the action of $\varphi'$ and $\sigma'$ for all cases gives the same model of the curve, 
$$F(X,Y,Z)=a_1X^5+a_2Y^5+a_3XY^2Z^2+a_4X^3YZ+a_2Z^5.$$
Let $P$ be the singular point of $C'$, then the point $P$ is fixed by $D_5$ and must be the only point fixed by $D_5$ in $\mathbb{P}^2$. Hence, $P=[1:0:0]$ and it follows that the model of the curve $C'$ is given by: 
$$F(X,Y,Z)=Y^5+a_3XY^2Z^2+a_4X^3YZ+Z^5.$$
Using MAGMA, we compute the normalization $C$ of $C'$ and then determine the number of $\mathbb F_q$-rational points of $C$ for each $q \in \{61,311,761,1061,1811,3911\}$, from which we deduce that there is no maximal trigonal curve over these fields.



 Now, suppose the curve  $C$ is given by a  complete intersection of three quadrics $Q_1, Q_2$ and $Q_3$ in $\mathbb P^4$. As above we want to use the information about the automorphism group of the curve to reduce the number of parameters of the quadrics.
 
Let $R,S \in \operatorname{Mat}_5(\mathbb{Z}[\gamma])$ be the  matrices generating $D_5$ in $\operatorname{Aut}(\mathcal{O}_K^5,h)$, where $\gamma = \frac{1+\sqrt{-19}}{2}$. An explicit description of these matrices appears in the appendix A. The minimal polynomial of $\gamma$ is $x^2-x+5$ and it splits over $\mathbb F_q$. Therefore, we can reduce $\gamma $ modulo a prime $\mathfrak{q}$ in $\mathcal{O}_K$ above $q$. This 
allows us to reduce $R$ and $S$ modulo $\mathfrak{q}$ to obtain the matrices $\overline{R},\overline{S} \in \operatorname{Mat}_5(\mathbb{F}_q)$. These reduced matrices generate the automorphism group $\operatorname{Aut}_{\mathbb{F}_q}(C)$ of the curve $C$.
Since $\mathbb F_q$ contains a primitive $5$-th root of unity $\zeta$, there is a basis of $\mathbb F^5_q$ where the matrix $\overline{R}$ has a diagonal form with powers of $\zeta$ in the diagonal entries and the matrix $\overline{S}$ is a matrix whose nonzero entries are also powers of $\zeta$. In  a fixed basis, the matrices are of the form:

$$\overline{R}=\begin{bmatrix}
     \zeta^3 & 0  & 0  & 0 & 0\\
     0  & \zeta^4  & 0 & 0 & 0\\
     0 & 0 & \zeta^2  & 0 & 0\\
     0 & 0 & 0 & \zeta  & 0\\
      0 & 0 & 0 & 0 & 1
\end{bmatrix}  \quad \overline{S}=\begin{bmatrix}
    0 & 0  & \zeta^2  & 0 & 0\\
     0  & 0  & 0 & \zeta & 0\\
     \zeta^3 & 0 & 0  & 0 & 0\\
     0 & \zeta^4 & 0 & 0  & 0\\
       0 & 0 & 0 & 0 & 1
\end{bmatrix}. $$ 

As the vanishing ideal the curve $C$ is generated by the three quadrics, the matrices $\overline{R}$ and $\overline{S}$ induce the automorphisms $\tilde{R}$ and $\tilde{S}$ of the same order, respectively, on the $\mathbb F_q$-projective space $\mathbb P(V)$ where $V=\langle  Q_1, Q_2, Q_3 \rangle$ is generated by the quadrics. In contrast to the trigonal case, here we do not only require the equation of the curve to remain invariant under the automorphisms $\overline{R}$ and $\overline{S}$, but we also want the vector space $V$ to remain stable under the action of $\tilde{R}$ and $\tilde{S}$.  Therefore, our goal is to determine the quadrics $Q_1$, $Q_2$ and $Q_3$ that satisfy these conditions. For this, we need to determine $\tilde{R}$ and $\tilde{S}$.

 The projective space $\mathbb P(V)$ is a  $3$-dimensional faithful 
 projective representation of $D_5$. To simplify the notation, we denote the projective space $\mathbb{P}(V)$ with $V$. In order to determine the model of the curve, we will use the theory developed in Section \ref{sec: proj} to study projective representations of $D_5$ over $\mathbb{F}_q$. Since projective representations of $D_5$ over $\mathbb F_q $ lead to classes in the cohomology group $H^2(D_5,\mathbb F_q^{\times})$, we will need the following proposition.

\begin{pro}\label{prop: h2}
Suppose that  $q\equiv 1\bmod 5$ and $D_5$ acts trivially on $\mathbb{F}_q^\times$, then $H^2(D_5,\mathbb{F}_q^\times)\cong C_2$, where $C_2$ is the cyclic group of order $2$.
\end{pro}

\begin{proof}
Let $r$ be the largest integer such that $2^r$ divides $q-1$.  We can write $\mathbb{F}_q^{\times}=C_{2^r}\times C_m$, where $m$ is coprime to $2$. Since $D_5$ acts trivially on $\mathbb{F}_q^\times$, then it acts trivially on $C_{2^r}$  and $C_m$ and we have 
$$ H^2(D_5,\mathbb{F}_q^\times)= H^2(D_5, C_{2^r})\oplus H^2(D_5, C_m).$$ Moreover, by \cite[p.164]{weibel1994introduction}, there is an exact sequence $$ \xymatrix{
          0\ar[r] &  \operatorname{Ext}^1(H_{1}(D_5,\mathbb{Z}), C_m) \ar[r]& H^2(D_5,C_m) \ar[r] & \operatorname{Hom}(H_2(D_5,\mathbb{Z}), C_m)\ar[r] &  0.       }$$ 
The group $H_2(D_5,\mathbb{Z})$ is trivial (see \cite[p.197]{weibel1994introduction}), and we have $H_{1}(D_5,\mathbb{Z})=D_5^{ab}=C_2$. Therefore, we obtain the isomorphism 
\begin{align*}
H^2(D_5, C_m) 
&\cong \operatorname{Ext}^1(C_2, C_m)\cong C_m / 2C_m=
\begin{cases}
0, & \text{if } \gcd(m,2)=1, \\
C_2, & \text{if } 2 \mid m.
\end{cases}
\end{align*} 
\end{proof}
By Remark \ref{rem aut} and Theorem \ref{zaytsev}, the group $\operatorname{Aut}_{\mathbb{F}_q}(C)=D_5$ acts trivially on $\mathbb{F}_q^\times$. It follows from Proposition \ref{prop: h2} that projective representations of $D_5$ over $\mathbb{F}_q$ can be attached to two possible classes in $H^2(D_5,\mathbb{F}_q^\times)$. The trivial class corresponding to liftable representations  and the non-trivial class corresponding to non-liftable representations of $D_5$. 

By  the proof of Proposition \ref{prop: h2}, we have the following:  $$H^2(D_5,\mathbb{F}_q^\times)=  H^2(D_5, C_{2^r})\cong C_2.$$ 
 This  tells us that any class of $H^2(D_5, C_{2^r})$ is the same as the the  class of  $H^2(D_5,\mathbb{F}_q^\times)$. Therefore, we can view the representative $\alpha$ of a  class in $H^2(D_5,\mathbb{F}_q^\times)$ as $2$-cocycle with values in $C_{2^r}$.
 
 Let $ (E_\alpha,p)$ be the unique (up to equivalence)  central extension of $D_5$ by $C_{2^r}$ corresponding to the class of $\alpha$ in $H^2(D_5, C_{2^r})$. Let 
 $$\varphi: C_{2^r}\longrightarrow \mathbb{F}_q^\times$$ 
 be the inclusion. 
 Then by Proposition \ref{prop:lift}, for any projective representation $\rho$ of $D_5$ with multiplier $\alpha$, there exists a linear representation
 $$T: E_\alpha\longrightarrow \operatorname{GL}(V)$$ 
 such that $\rho\circ p=\Pi\circ T$, where $\Pi:  \operatorname{GL}(V)\longrightarrow \operatorname{PGL}(V)$ and $p: E_\alpha\longrightarrow D_5$.  
 
 Hence, any projective representation of $D_5$ over $\mathbb{F}_q$ with multiplier $\alpha$ is realized as linear representation (up to a central scalar) of the unique (up to equivalence) central extension $E_\alpha$.
\begin{pro}
Let $r$ be any positive integer. A
 central extension of $D_5$  by $C_{2^r}$ is given by $$\widetilde{D_5}^r:= \{a, b\,|\, a^5=1,\, b^{2^{r+1}}=1,\, bab^{-1}=a^{-1}\}.$$ 
\end{pro}
\begin{proof}
Since $bab^{-1}=a^{-1}$, we have $b^2ab^{-2}=bbab^{-1}b^{-1}=ba^{-1}b^{-1}=(bab^{-1})^{-1}=a$. Therefore, $b^2$ commutes  with $a$. Hence, $b^2$ belongs to the center of $\widetilde{D_5}^r$. Furthermore, since $b$ has order $2^{r+1}$ and the order of $b^2$ divides $2^r$, it follows that $b^2$ has order $2^r$. Thus, $C_{2^r}=\langle b^2 \rangle$. Finally, we have $\widetilde{D_5}^r/C_{2^r}=\{\bar{a}, \bar{b}\,|\, \bar{a}^5=1,\, \bar{b}^{2}=1,\, \bar{b}\bar{a}\bar{b}^{-1}=\bar{a}^{-1}\}= D_5$.
\end{proof}
The $3$-dimensional projective representation $(\rho, V)$  of $D_5$ can be lifted to a linear representation of $\widetilde{D_5}^r$.

To determine all irreducible representations, let $N$ be the subgroup of $\widetilde{D_5}^r$ generated by $a$ and $b^2$. Since  $b^2$ commutes with $a$ and they have coprime orders, it follows that  $N$ is an abelian subgroup given by  $$N\cong C_5\times C_{2^r}.$$

Furthermore, since the order of $\widetilde{D_5}^r$ is $2^{r+1}\cdot5$, then  $N$ has index two in $\widetilde{D_5}^r$. Hence,  all irreducible linear representations  of $\widetilde{D_5}^r$ have dimension $1$ or $2$ (see \cite[p.61]{SerreLR}).

It follows that we can write $\rho=\chi\oplus \rho_2$, where $\chi$ is a $1$-dimensional representation and $\rho_2$ is an irreducible two-dimensional representation, or we can write $\rho$ as a sum of $1$-dimensional representations. 

We start by analyzing the one-dimensional representation $\chi$. We have $$\chi(a^5)=\chi(a)^5=1 \quad\text{ and }\quad \chi(b)\chi(a)\chi(b)^{-1}=\chi(a)=\chi(a)^{-1}.$$ 
This implies that $\chi(a)=1$ and from the relation $b^{2^{r+1}}=1$, we obtain $\chi(b)^{2^{r+1}}=1.$ Hence $\chi(b)$ is a root of unity in $\mathbb{F}_q^{\times}$ of order dividing $2^{r+1}$. Let $\lambda\in \mathbb{F}_q^{\times}$ be an element of order dividing $2^r$ such that $\chi(b)^2=\lambda$. 
It follows that 
$$\chi(b)=\pm\nu $$
where $\pm\nu\in \overline{\mathbb{F}}_q$ are the roots  of the polynomial $x^2-\lambda$.   

For the two-dimensional representation $\rho_2$, since $\mathbb{F}_q$ contains a 
primitive $5$th root of unity $\zeta$, we choose a basis such that $$\rho_2(a)= \begin{pmatrix}
\zeta^k & 0 \\ 0  &\zeta^{-k} \end{pmatrix},$$ where $k\in \{1,2\}$. Suppose $b$ acts by a matrix $B$, we have 
$$B^2=\lambda I_2 \quad \text{and}\quad B\begin{pmatrix}
\zeta^k & 0 \\ 0  &\zeta^{-k} \end{pmatrix} B^{-1}= \begin{pmatrix} \zeta^{-k} & 0 \\ 0  &\zeta^{k} \end{pmatrix}.$$
This implies that $$B=\begin{pmatrix}0 & t\\z & 0
\end{pmatrix}$$ with $zt=\lambda$. 
Since $\nu^2=\lambda$, up to a conjugation (by $P=\begin{pmatrix}1 & 0\\ 0 & s \end{pmatrix}$ with $s=\nu/t$), we can take
$$B=\begin{pmatrix}
    0 & \nu\\ \nu & 0
\end{pmatrix}\quad \text{or} \quad B=\begin{pmatrix}0 & -\nu\\ -\nu & 0 \end{pmatrix}.$$ 
   
We conclude that if $\rho=\chi\oplus \rho_2$, then $$\tilde{R}=\rho(a)=\begin{pmatrix}
    1 & 0 & 0\\
    0 & \zeta^k & 0\\
    0 & 0 & \zeta^{-k}
\end{pmatrix}$$ and
$$\tilde{S}=\rho(b)=\begin{pmatrix}
    \pm\nu & 0 & 0\\
    0 &  0 & \nu\\
    0 & \nu & 0
\end{pmatrix} \text{ or }\tilde{S}=\rho(b)=\begin{pmatrix}
    \pm\nu & 0 & 0\\
    0 &  0 & -\nu\\
    0 & -\nu & 0
\end{pmatrix}.$$ 

If $\rho$ is a sum of 1-dimensional representations then $$\tilde{R}=\rho(a)=\begin{pmatrix}
    1 & 0 & 0\\
    0 & 1 & 0\\
    0 & 0 & 1
\end{pmatrix} \quad\text{ and }\quad\tilde{S}=\rho(b)=\begin{pmatrix}
    \pm\nu & 0 & 0\\
    0 &  \pm \nu & 0\\
    0 & 0 & \pm\nu
\end{pmatrix}$$

The case where $\rho$ is a sum of 1-dimensional representations leads to a contradiction since $\tilde{R}$ has order $5$.  We therefore consider the case $$\rho=\chi\oplus \rho_2.$$
Under this decomposition, the action of  $\overline{R}$ gives,
$$Q_1(\overline{R}(x,y,z,u,v))=Q_1,\quad\, Q_2(\overline{R}(x,y,z,u,v))=\zeta^k Q_2, \quad \text{ and }\quad Q_3(\overline{R}(x,y,z,u,v))=\zeta^{-k}Q_3.$$
Similarly, under the action of $\overline{S}$, we have 
$$Q_1(\overline{S}(x,y,z,u,v))=\pm\nu Q_1,\quad\, Q_2(\overline{S}(x,y,z,u,v))=\nu Q_3\quad \text{ and }\quad Q_3(\overline{S}(x,y,z,u,v))=\nu Q_2.$$

It follows that $Q_1$ can exist only if $\nu=\pm1$. Indeed, for $k=1$ the condition imposed by $\overline{R}$ gives that $Q_1=a_1XZ+a_2YU+a_3V^2$. The condition imposed by $\overline{S}$ gives us that $a_1=a_2=a_3=0$ unless  $\nu=\pm1$.   

However, $\nu=\pm1$ implies that $\lambda=1$ and therefore,  $\widetilde{D_5}=D_5$. This tells us that 
$V$ is a projective representation corresponding to the trivial class in $H^2(D_5,\mathbb F_q^{\times})$. Thus, we can  view it as a linear representation of $D_5$. 
 So, we have 
  

\begin{itemize}
    \item Case 1. $\tilde{R}=\begin{bmatrix}
    1 & 0  & 0  \\
     0  & \zeta &0\\
    0 & 0 & \zeta^4\\
     
\end{bmatrix}, \quad \tilde{S}=\begin{bmatrix}
    \pm1 & 0  & 0\\
     0  & 0 & \zeta^4 \\
      0 & \zeta &  0\\
     
\end{bmatrix}$, 
\vspace{0.3cm}
\item  Case 2. $\tilde{R}=\begin{bmatrix}
    1 & 0  & 0  \\
     0  & \zeta^2 &0\\
     0 & 0 & \zeta^3\\
     
\end{bmatrix}, \quad \tilde{S}=\begin{bmatrix}
    \pm1 & 0  & 0\\
     0  & 0 & \zeta^3 \\
      0 & \zeta^2 &  0\\
     
\end{bmatrix}.$

\end{itemize}

To obtain the models of the curve we apply the action of $\overline{R}$ and $\overline{S}$ on the variables and see how the quadrics should be transformed looking at the actions of $\tilde{R}$ and $\tilde{S}$ on the quadrics. That is, $$\begin{pmatrix}
    Q_1(\overline{R}(x,y,z,u,v))\\
    Q_2(\overline{R}(x,y,z,u,v))\\
    Q_3(\overline{R}(x,y,z,u,v))
\end{pmatrix}=\tilde{R}\begin{pmatrix}
    Q_1\\ Q_2\\Q_3
\end{pmatrix}\quad \text{and}\quad \begin{pmatrix}
    Q_1(\overline{S}(x,y,z,u,v))\\
    Q_2(\overline{S}(x,y,z,u,v))\\
    Q_3(\overline{S}(x,y,z,u,v))
\end{pmatrix}=\tilde{S}\begin{pmatrix}
    Q_1\\ Q_2\\Q_3
\end{pmatrix}.$$ Therefore, we obtain the following models of the potential curve:

\begin{itemize}
    \item Case 1. $\begin{array}{cc}
         &  \\
         Q_1= a_1XZ+a_2YU+a_3V^2\\
         Q_2= a_4YZ+a_5UV+a_6X^2\\
         Q_3= a_4XU+a_5YV+a_6Z^2\\
         
    \end{array}$, 
          
    \item Case 2. $\begin{array}{cc}
          &\\
        Q_1= a_1XZ+a_2YU+a_3V^2\\
        Q_2= a_4XY+a_5ZV+a_6U^2\\
        Q_3= a_4XV+a_5ZU+a_6Y^2 
    \end{array}$.
\end{itemize} 

The families still involve too many free parameters to allow for an exhaustive computational search. However, class field theory approach yields a fast computational search, see Appendix B.





\section{ The case $q\equiv 0\bmod 5$} \label{sec: q=0[5]}
In this section, we prove that $q=5^7$ is the only value satisfying $q\equiv 0 \bmod 5$  and $d(\mathbb{F}_q)=-19$. To establish this, we rely on the following lemmas.

\begin{lem}[Chap 4, Lemma 5.2 \cite{Cassels_1986}]\label{lem: powseries}
Let $a\in \mathbb{Q}_p$ and suppose that $|a|_p\leq 2^{-2}$ if $p=2$ or $|a|_p\leq p^{-1}$ otherwise. Then there is a power series $\phi_a(x)=\sum_{m=0}^{\infty}\gamma_m x^m$ where, $\gamma_m\in \mathbb{Q}_p$, $\gamma_m\to 0$ such that  $(1+a)^s=\phi_a(s)$ for all $s\in \mathbb{Z}$.
\end{lem}

\begin{lem}\label{recurence relation}
    Let $u_n$ be defined by $u_0=0,u_1=1$ and 
    $$u_n=u_{n-1}-5u_{n-2}, n\geq2.$$
    Then $u_n= 1$ only for $n=1,2,7$.
\end{lem}
\begin{proof}
 The first values of $u_n$ are 
 
\begin{center}
  \begin{tabular}{c|ccccccccccc}
      $n$&0 & 1& 2&3 &4 & 5 & 6& 7&8&9 & 10 \\
      \hline
      $u_n$& 0&1&1 &-4 &-9& 11 & 56& 1&-279&-284 & 1111
 \end{tabular}
    
\end{center}
 The characteristic equation of the recurrence is given by $x^2-x+5=0$. Let $\alpha$ and $\beta$ be roots of $f(x)=x^2-x+5$ in the field of complex numbers  $\mathbb{C}$. Then the general solution to the recurrence relation  is 
$$u_n=\lambda_1\alpha^n+\lambda_2\beta^n,$$
where $\lambda_1=\frac{1}{\alpha-\beta}$ and $\lambda_2=\frac{-1}{\alpha-\beta}$ are determined by the initial conditions.  Substituting $\lambda_1$ and $\lambda_2$ we obtain $$u_n=\frac{\alpha^n-\beta^n}{\alpha-\beta}.$$
 
Instead of considering roots of $f$ in $\mathbb{C}$, we work with the roots $\alpha$ and $\beta$ in $\mathbb{Q}_p$, where $p$ is a prime number for which $f$ splits in $\mathbb{Q}_p[x]$, for instance $p=7$.  In this setting, we will write $u_n-1=g(s)$, where  $s$ depends on $n$ and $g(s)=\sum_{k=0}^{\infty} g_k s^k$ with $g_k\to 0$ as $k\to \infty$. Then, we will use Strassmann's theorem \cite[Chapter 4, Theorem 4.1]{Cassels_1986} to  bound the number of roots of $g$ and subsequently determine the values of $n$ for which $u_n-1=0$.

Since  $f(2)\equiv 0 \bmod 7$, by Hensel's lemma, the root $\alpha$ in $\mathbb{Q}_7$ satisfies
$$ \alpha \equiv 2 \bmod 7,\quad \alpha \equiv 16 \bmod 7^2,\quad  \alpha \equiv 163 \bmod 7^3 .$$

The other root $\beta$ satisfies  $$ \beta\equiv 6 \bmod 7, \quad \beta\equiv 34 \bmod 7^2, \quad  \beta\equiv 181 \bmod 7^3$$

We want to apply the Lemma \ref{lem: powseries} to $\alpha$ and $\beta$, however, they do not satisfy the condition of the lemma. Instead, we will apply the lemma to 

$$ a:= \alpha^6-1\equiv 0\bmod 7\quad \text{ and }\quad  b:= \beta^6-1\equiv 0\bmod 7.$$
One remarks that for any positive integers $r$ and $s$, 

\begin{equation}\label{eq: rec}
  u_{r+6s} = \frac{\alpha^{r+6s} - \beta^{r+6s}}{\alpha - \beta}=\frac{\alpha^r (1+a)^s- \beta^r(1+b)^s}{\alpha - \beta}  \equiv u_r \bmod 7.  
\end{equation}

As we can always write $n=r+6s$ with $0\leq r\leq 5 $, this tells us that, to  find all $n$ such that $u_n= 1$, we need to find all $s$ such that $ u_{r+6s}=1$ for those fixed values of $r$. More precisely, only for $r=1,2$ since for $r=0,3,4,5$ and any $s$, we have that $u_{r+6s}\neq1$ because $u_r\not \equiv 1 \bmod 7$. For the rest of the proof, we restrict to $r\in\{1,2\}$.

From  (\ref{eq: rec}), we  have 
$$ (\alpha-\beta)(u_{r+6s}-1)=\alpha ^r(1+a)^s-\beta^r(1+b)^s - (\alpha -\beta)=g_r(s). $$ 
Furthermore,  $a$ and $b$ satisfy the hypothesis of the Lemma \ref{lem: powseries},  then $g_r(s)$ satisfies the hypothesis of Strassmann's theorem[cite].  We  write $g_r(s)$ as follows:
\begin{align*}
  g_r(s)&=\alpha^r-\beta^r -(\alpha-\beta)+ \alpha^ras-\beta^rbs+ \alpha^ra^2\frac{s(s-1)}{2}- \beta^rb^2\frac{s(s-1)}{2}+\ldots\\
  &= \sum^\infty_{k=0}c_{k,r}s^k.
\end{align*}
It is clear that $c_{0,r}=g_r(0)=(\alpha-\beta)(u_{r}-1)=0$. Moreover, as for any integer $k$, we have  $a^{k}\equiv 0 \bmod 7^{k}$ and $b^{k}\equiv 0 \bmod 7^{k}$, then    $c_{k,r}\equiv \alpha^ra^k-\beta^rb^k\bmod 7^{k+1}$.   In particular, 
 $$ c_{k,r}\equiv 0\bmod 7 \,\text{  for all } k,\quad  c_{1,r}\equiv \alpha^ra-\beta^rb\bmod 7^2 \quad\text{ and }\quad c_{2,r}\equiv \alpha^ra^2-\beta^rb^2\bmod 7^3.$$ 
We have  $c_{1,2}\not \equiv 0\bmod 7^2$ and  $c_{k,2}\equiv 0\bmod 7^2$ for all $k\geq 2$. Therefore, by Strassmann's theorem the equation $g_2(s)=0$ has at most one solution in $\mathbb{Z}_7$. Thus $s=0$ is the only solution, so $u_{2+6s}\neq 1$ for all $s>0$.

As $c_{k,1}\equiv 0\bmod 7^2$ for all $k$,  we verify that $c_{2,1} \not \equiv 0\bmod 7^3$. Moreover, for all $k\ge 3$, $c_{k,1}\equiv 0\bmod 7^3$, then the equation $g_1(s)=0$ has at most two solutions in $\mathbb{Z}_7$. Observing that $s=0$ and $s=1$ are solutions, we conclude that these are the only solutions. 

Thus, the pairs $(r,s)$ for which $g_r(s)=0$ are  $(1,0), (1,1),$ and $(2,0)$. They correspond to $n= 1, 7$ and $2$, respectively.
\end{proof}

\begin{rem}\label{un=-1}
As $u_r \not\equiv -1 \bmod 7$ for all $r\in \{0,1,2,3,4,5\}$, it follows that $u_{r+6s}\neq -1$ for all $r,s$, hence $u_n\neq-1$ for all $n$.
\end{rem}
\begin{pro}\label{Diophantine}
    The only value of $q$ satisfying $q\equiv 0 \bmod 5$  and $d(\mathbb{F}_q)=-19$ is $q=5^7$. 
\end{pro}

\begin{proof}
Recall that the discriminant of the finite field $\mathbb{F}_q$ is given by 
$d(\mathbb{F}_q) = \lfloor 2\sqrt{q} \rfloor^2 - 4q$. Then for 
 $q\equiv 0 \bmod 5$  and $d(\mathbb{F}_q)=-19$, we obtain  the Diophantine equation 
    $$x^2+19=4\cdot5^n,$$
where $x$ and $n$ are unknown. From the equation, we deduce  that $x$ must be odd. That is,  $x=2y-1$ for some integer $y\in \mathbb Z$. Then we have 
$$4\cdot 5^n= x^2+19=(2y-1)^2+19=4(y^2-y+5).$$
 It follows that 
$$y^2-y+5=(y-\alpha)(y-\beta)=5^n,$$
where $\alpha$ and $\beta$ are the roots in $\mathbb{C}$ of $y^2-y+5=0$.
Furthermore, since  $\alpha\beta=5$, we obtain  $$(y-\alpha)(y-\beta)=\alpha^n \beta^n.$$
The above factorization can be viewed in the ring 
 $\mathbb Z[\alpha]$ which is the ring of integers of $\mathbb Q(\sqrt{-19})$ and it happens to be a unique factorization domain.   Note that the ideals $(\alpha)$,  $(\beta)$ are coprime.  Hence,
 
 $$y-\alpha=\epsilon \beta^n, \quad  y-\beta=\epsilon \alpha^n, \text{ where } \epsilon \in \{-1,1\}.$$
Thus,  $$\alpha - \beta = \epsilon (\alpha^n-\beta^n). $$ 
From Lemma \ref{recurence relation} and Remark \ref{un=-1}, we know $u_n=\frac{\alpha^n-\beta^n}{\alpha-\beta}=\pm 1$ only for $n=1,2,7$.  Therefore, the only  solutions $(x,n)$ to the diophantine equation are  $(1,1)$, $(9,2)$, and $(559,7)$.  Furthermore, the only solution with $x = \lfloor 2\sqrt{5^n} \rfloor$ is  $(559,7)$.
\end{proof}

 In the Appendix B, Schoof proved that for $q=5^7$ the curve does not exists.

\section*{Acknowledgments}
The authors thank the Centre International de Mathématiques Pures et Appliquées (CIMPA) for their support through the organization of the CIMPA school in Turkey in 2022, where this project started. We also thank the Centre International de Rencontres Mathématiques (CIRM) for providing an excellent research enviroment during our stay. We thank  Christophe Ritzenthaler for  proposing this project and for his insightful discussions, Gabor Wiese for reading the manuscript and providing valuable feedback, and René Schoof for his insight on the case $q\equiv0\bmod 5$. The second-named author is supported by the Luxembourg National Research Fund AFR-PhD 16981197. 

\section*{Appendix A}\label{appendix}
Here, $H_1$ is a genus of an unimodular irreducible hermitian module of dimension $5$ over an imaginary quadratic extension  $K$  of $\mathbb{Q}$ with discriminant $d(K) = -19$. Perhaps see \cite{hermitian}. In addition, we give the size of its automorphism group. The automorphism group of the  hermitian module is given by the matrices $U\in \operatorname{GL_5(\mathcal{O}_K)}$ such that $\overline{U}^TH_1U=H_1$.  The matrices $R$ and $S$ below represent the generators of the subgroup  $D_5$ in the automorphism group. One can obtain $R$ and $S$  with MAGMA from \cite{HermitianLattices10} with the following commands: 
\begin{verbatim}
g:=5;                         
dis:=-19; 
LList:=OrderToLattices(dis,g);
LList[10];
\end{verbatim}

Note that in \cite{hermitian}, two other genera, $H_2$ and $H_3$, are described, each with an automorphism group of the same size. However, they will give conjugates of $R$ and $S$.

$$
H_1=\begin{bmatrix}
   3 &  &  &  &  \\
   \gamma & 3 &  &  &  \\
   0 & -1 & 3 &  &  \\
  -1 & 0 & 1 & 4 &  \\
  1-\gamma & -\gamma & -1+\gamma & 1+\gamma & 5
\end{bmatrix},
\quad |\operatorname{Aut}| = 2^2\times 5,$$

\[
R=\begin{array}{cc}
\begin{bmatrix}
-\gamma - 1 & 3\gamma - 7 & -2\gamma & 4 & 0 \\
\gamma + 1 & 2\gamma + 1 & -2 & 0 & 4 \\
\gamma - 2 & 4\gamma - 8 & -2\gamma - 1 & 2 & 2\gamma + 2 \\
-\gamma + 4 & -\gamma - 6 & -\gamma + 4 & 2\gamma + 1 & -\gamma + 2 \\
0 & 6 & \gamma - 2 & -\gamma & -\gamma
\end{bmatrix},
&
S=\begin{bmatrix}
-1 & -\gamma - 1 & 2 & 0 & 0 \\
0 & 1 & 0 & 0 & 0 \\
0 & 0 & 1 & 0 & 0 \\
0 & -\gamma + 3 & \gamma & 0 & 1 \\
0 & \gamma - 3 & -\gamma & 1 & 0
\end{bmatrix}
\end{array}.
\]


\input{App}

\let\cleardoublepage\clearpage

\bibliographystyle{plainurl}
\bibliography{Genus_5}

\end{document}

%% file: App.tex


\newcommand{\QQ}{\mathbf{Q}}
\newcommand{\ZZ}{\mathbf{Z}}
\newcommand{\FF}{\mathbf{F}}
\newcommand{\HH}{\mathbf{H}}
\newcommand{\RR}{\mathbf{R}}
\newcommand{\CC}{\mathbf{C}}
\newcommand{\PP}{\mathbf{P}}

\renewcommand{\AA}{\mathbf{A}}

\newcommand{\abs}[1]{\lvert\!\lvert #1 \rvert\!\rvert}
\newcommand{\zmod}[1]{\,\,(\mathrm{mod}\,\,#1)}
\newcommand{\Gal}{\mathrm{Gal}}
\newcommand{\pp}{\mathfrak{p}}
\newcommand{\qq}{\mathfrak{q}}
\newcommand{\ga}{J_{\pi}}
\newcommand{\gb}{J_{\pi'}}
\newcommand{\onto}{\twoheadrightarrow}

\newcommand{\Sref}{1}
\newcommand{\Zref}{2}

\newtheorem{proposition}{Proposition}
\renewcommand{\theproposition}{B.\arabic{proposition}}

\newtheorem{corollary}[proposition]{Corollary}




\section*{Appendix B to ``Maximal curves of genus $5$ over finite fields"} \label{appendix B}

Let $X$ be a smooth, projective, absolutely irreducible curve over a finite field
$\FF_q$ of genus $g_X=5$. Suppose that $X$ admits an automorphism of order $5$
and let $G$ denote the group generated by it. Let $E$ denote the quotient of $X$
by $G$. The conductor-discriminant formula applied to the cyclic Galois cover
$X \rightarrow E$ is
\[2g_X - 2 = 5(2g_E - 2)+\sum_{\chi} \deg(\mathrm{conductor}(\chi)).
\]
It easily implies that the genus $g_E$ of $E$ is $1$ and that the degree of the
common conductor $D$ of the non-trivial characters $\chi$ of $G$ is equal to $2$.

By class field theory the reciprocity homomorphism
\[
\theta : \AA_E^*/K^*U_D \longrightarrow G
\]
is surjective. Here $K$ is the function field $\FF_q(E)$ and $\AA_E$ is the
ad\`ele ring of $E$. The subgroup of the id\`ele group $\AA_E^*$ of unit
id\`eles is denoted by $U$. The quotient group $\AA_E^*/U$ is naturally
isomorphic to the divisor group $\mathrm{Div}(E)$. Let $U_D$ denote the subgroup
of $u \in U$ for which $u \equiv 1 \zmod D$. There is an exact sequence
\begin{equation*}
 \begin{matrix}
0 & \longrightarrow & U/\FF_q^*U_D
& \longrightarrow & \AA_E^*/K^*U_D
& \longrightarrow & \AA_E^*/K^*U
& \longrightarrow & 0 \\[0.3em]
&&&&&& \Vert && \\[0.3em]
&&&&&& \mathrm{Pic}(E) &&
\end{matrix}   \tag{$\star$}
\end{equation*}

Since the degree of $D$ is $2$, there are three possibilities for $D$. It is
either a point of degree $2$, or it is the sum of two degree $1$ points, which
may or may not be equal. The group $U/\FF_q^*U_D$ is isomorphic to
$\FF_{q^2}^*/\FF_q^*$, $\FF_q^*$ or $\FF_q$ in these cases. The orders of these
groups are $q+1$, $q-1$ and $q$, respectively.

\begin{proposition}
We have $q \not\equiv \pm 2 \zmod 5$.
\end{proposition}
\begin{proof}
If $q$ were congruent to $\pm 2 \zmod 5$, none of the possible orders of
$U/\FF_q^*U_D$ is divisible by $5$. Therefore $U/\FF_q^*U_D$ is in the kernel of
$\theta$. It follows that $\theta$ factors through
$\AA_E^*/K^*U = \mathrm{Pic}(E)$. But that implies that the cover
$X \rightarrow E$ is unramified, which it is not. This contradiction proves the
proposition.
\end{proof}

By Zaytsev~\cite[Thm.~4.13]{zaytsevoptimal}, an optimal genus $5$ curve over $\FF_q$
for which one has $[2\sqrt{q}]^2 - 4q = -19$, admits an automorphism of order $5$. Therefore Proposition~B.1 applies. It gives an alternative proof of Theorem \ref{Theorem 3.2} of this paper.

\vspace{0.5cm}

The next proposition regards $q=5$. Over $\FF_5$ there is a unique elliptic
curve $E'$ with $\#E'(\FF_5)=5$. The trace of its Frobenius endomorphism is $1$
and the discriminant of its endomorphism ring is $-19$. The curve $E'$ is given
by the Weierstrass equation $y^2 = x^3 - 2x + 2$.

\begin{proposition}
If $q=5$, then the curve $E$ cannot be isomorphic to $E'$.
\end{proposition}

\begin{proof}
If $q=5$, the degree $2$ divisor $D$ can only be $2P$ for some
$\FF_q$-rational point $P$ of $E$. In this case the order of
$U/\FF_q^*U_D$ is equal to $q=5$. The exact sequence $(\star)$ modulo fifth powers
leads to the exact sequence
\[
U/\FF_5^*U_D
\longrightarrow
\AA_E^*/{\AA_E^*}^5K^*U_D
\longrightarrow
\AA_E^*/{\AA_E^*}^5K^*U
\longrightarrow 0.
\]
If $E$ were isomorphic to the elliptic curve $E'$ given by
$y^2=x^3-2x+2$, the leftmost homomorphism is zero. In other words, $U$ is
contained in the subgroup ${\AA_E^*}^5K^*U_D$ of $\AA_E^*$. To see this, we first observe that the group $E(\FF_5)$ acts transitively on
itself by translations. Therefore we may take for $P$ any of the five points in
$E(\FF_5)$. We choose $P=(1,1)$. Then we put $Q=(2,-1)$ and let
$\xi \in \AA_E^*$ denote an id\`ele whose image in $\mathrm{Div}(E)$ is the
divisor $Q-O$, where $O$ is the point of $E$ at infinity. It is not difficult to see that the divisor of the function
\[
h =
\frac{(y+1)^2(x-2)}{y-2x-2}
\]
is equal to $5(Q-O)$. Therefore both $\xi^5$ and $h$ are elements of
$\AA_E^*$ whose images in $\mathrm{Div}(E)$ are equal to $5(Q-O)$. It follows
that $v = \frac{\xi^5}{h}$ is in $U$. Since $h(P)=2$ and
\[
h-2 =
\frac{(x-1)(2y+x^3-x^2+2x)}{y-2x-2}
\]
has a simple zero in $P$, the id\`ele unit $v$ generates $U/U_D$. Therefore
every $u \in U$ can be written as $v^k w$ for some $k \in \ZZ$ and
$w \in U_D$. Since $v$ is in ${\AA_E^*}^5K^*$, this proves that we indeed have $U \subset {\AA_E^*}^5K^*U_D$.

It follows that the natural map
\[
\AA_E^*/{\AA_E^*}^5K^*U_D
\mathop{\longrightarrow}^{\cong}
\AA_E^*/{\AA_E^*}^5K^*U
\]
is an isomorphism. By class field theory, the group on the left is isomorphic
to the Galois group of the maximal abelian exponent $5$ extension $L$ of $K$ of
conductor at most $D=2P$. Note that the function field $\FF_5(X)$ of the curve
$X$ is a subfield of $L$. Since $\AA_E^*/{\AA_E^*}^5K^*U_D$ is isomorphic to  $\AA_E^*/{\AA_E^*}^5K^*U$,
 the extension $K \subset L$ is unramified at \textit{all} places. Since we have $K \subset \FF_5(X) \subset L$,
 the same is true for $\FF_5(X)$. However, the cover $X \rightarrow E'$ is
actually ramified and hence we obtain a contradiction.
This proves the proposition.
\end{proof}

\begin{corollary}
When $q$ is a power of $5$, there is no optimal genus $5$ curve $X$ over
$\FF_q$ for which $[2\sqrt{q}]^2 - 4q = -19.$
\end{corollary}

\begin{proof}
Suppose $X$ is such an optimal curve. By Proposition \ref{Diophantine} of this paper, the
curve $X$ cannot exist, except possibly when $q=5^7$. In this case we have $[2\sqrt{q}] = 559$ and the eigenvalues of the characteristic polynomial of Frobenius are $\frac{-559+\sqrt{-19}}{2} = \phi^7$ and its complex conjugate, where $\phi = \frac{1+\sqrt{-19}}{2}$. Since the ring $\ZZ[\phi]$ is generated by $\phi^7$, a theorem of Serre~\cite[Thm.~9 of the appendix]{Lauter} implies that $X$ is the base change
of a genus $5$ curve $X'$ over $\FF_5$.

By Zaytsev~\cite[Thm.~4.13]{zaytsevoptimal}, the automorphism group of $X$ over
$\FF_{5^7}$ is isomorphic to the dihedral group $D_5$. The Galois group of
$\FF_{5^7}$ over $\FF_5$ acts on $\mathrm{Aut}_{\FF_{5^7}}(X')$. Since $7$ is
prime to $\#\mathrm{Aut}(D_5)=20$, this action is trivial. Therefore all
automorphisms of $X'$ are defined over $\FF_5$. The quotient of $X'$ by an
automorphism of order $5$ is isomorphic to the unique elliptic curve $E'$ over
$\FF_5$ with five rational points. Therefore Proposition B.2 applies and we
conclude that this is not possible. Therefore $X$ cannot exist.
\end{proof}

In view of the results in this paper, the fields $\FF_q$ over which an optimal
genus $5$ curve $X$ with $[2\sqrt{q}]^2 - 4q = -19$ may exist necessarily satisfy $q \equiv 1 \zmod 5$. As we explained above, in these cases $X$ is a cyclic degree $5$ cover ramified
over two distinct $\FF_q$-rational points of a specific genus $1$ curve $E$. By
translating we may assume that the two points are the point $O$ at infinity and
a second point $P$. By Kummer theory the function field of $X$ is of the form $\FF_q(E)\left(\sqrt[5]{h}\right)$, where $h \in \FF_q(E)$ is a function whose divisor is $m(P-O)$ with
$m=\#E(\FF_q)$. The function $h$ is unique up to a scalar multiple. The number
of $\FF_q$-rational points on $X$ is equal to $2+5r$, where $r$ is the number
of points $Q \in E(\FF_q) - \{P,O\}$ for which $h(Q)$ is a fifth power in $\FF_q^*$.

A short Pari-GP program, computing the number $r$ for each point $P \neq O$ in
$E(\FF_q)$ and each function $h$, showed that for $q<10000$ there is no optimal
genus $5$ curve $X$ over $\FF_q$ with $[2\sqrt{q}]^2 - 4q = -19$. This is actually a short list of ten primes, the smallest and largest being
$q=61$ and $q=9511$, respectively. The entire computation took less than two
hours.

\vspace{1em}

\noindent
Ren\'e Schoof\\[0.5em]
Dipartimento di Matematica\\
$2^{\mathrm{a}}$ Universit\`a di Roma ``Tor Vergata''\\
I-00133 Roma, Italy\\
Email: \texttt{schoof@mat.uniroma2.it}
